\documentclass[12pt, a4paper]{article}

\usepackage[T1]{fontenc}
\usepackage{amsmath,amssymb,amsfonts,amsthm}
\usepackage{graphicx}
\usepackage[colorlinks=true, allcolors=blue]{hyperref}
\usepackage{mathrsfs}
\usepackage{mathtools} 
\usepackage{changes}
\usepackage{enumitem}
\usepackage[left=2.5cm, right=2.5cm, top=2.5cm, bottom=2.5cm]{geometry}
\newtheorem{theorem}{Theorem}[section]
\newtheorem{lemma}[theorem]{Lemma}

\newtheorem{proposition}[theorem]{Proposition}
\newtheorem{corollary}[theorem]{Corollary}
\newtheorem{remark}[theorem]{Remark}
\def\p{\partial}
\def\bm{\boldsymbol}
\def\ve{\varepsilon}
\def\q{\quad}
\def\cd{\cdot}
\def\O{\Omega}
\def\na{\nabla}
\numberwithin{equation}{section} 

\providecommand{\keywords}[1]
{
\small
\textbf{Keywords:} #1
}

\begin{document}

\title{\Large\bf Lifespan of the  Non-resistive Hall-MHD System with Small Magnetic Gradient}

\author{\normalsize\sc Linbin Yang and Taoran Zhou}

\date{}

\maketitle

\begin{abstract}
In this paper, we consider the non-resistive axially symmetric Hall-MHD system. We show that the lifespan of their strong solutions can be arbitrarily large if their initial magnetic gradient are small enough. Precise lifespan lower bounds for both viscid and inviscid cases are given. 
\end{abstract}
\keywords{Non-resistive, Hall-MHD equations, Axially symmetric, Small magnetic gradient, Lifespan}\\
\\
\textbf{Mathematics Subject Classification:} 35Q35 $\cdot$ 76B03 $\cdot$ 76D03

\section{Introduction and main results}
The 3D non-resistive Hall-MHD (HMHD) system reads
\begin{equation}\label{eq1.1}\tag{HMHD}
    \left\{\begin{array}{l}
		\partial_t \boldsymbol{u}+\boldsymbol{u} \cdot \nabla \boldsymbol{u}+\nabla p-\mu \Delta \boldsymbol{u}=\frac{1}{\mu_0} \boldsymbol{h} \cdot \nabla \boldsymbol{h}, \\
		\partial_t \boldsymbol{h}+\boldsymbol{u} \cdot \nabla \boldsymbol{h}+\nu_0 \nabla \times[(\nabla \times \boldsymbol{h}) \times \boldsymbol{h}]=\boldsymbol{h} \cdot \nabla \boldsymbol{u},  \\
		\nabla \cdot \boldsymbol{u}=0, \\
		\nabla \cdot \boldsymbol{h}=0 .
		\end{array}\right.
	\end{equation}
	with initial data
\begin{equation}\label{ini-data}
	\boldsymbol{u}(0,x)=\boldsymbol{u}_0(x),\text{ and }\, \boldsymbol{h}(0,x)=\boldsymbol{h}_0(x).
\end{equation}
Here $(\boldsymbol{u}, \boldsymbol{h}): \mathbb{R}^+\times\mathbb{R}^3 \rightarrow \mathbb{R}^3\times\mathbb{R}^3$ is the velocity and the magnetic field, respectively. $p: \mathbb{R}^3 \rightarrow \mathbb{R}$ represents the pressure. $\mu, \mu_0, \nu_0$ stand for the constant viscosity, vacuum permeability and ratio for the Hall effect.
	
The equations of Magnetohydrodynamics (MHD) are widely used in describing many physical phenomena, such as fusion plasmas and star formation(see \cite{forbes1991magnetic,shalybkov1997hall}), as they play a crucial role in describing the large-scale interaction between conducting fluids and magnetic fields. The Hall-MHD differs from the classical incompressible MHD, due to the Hall term $\nabla \times [(\nabla \times \boldsymbol{h})\times \boldsymbol{h}]$ which accounts for the decoupling of electron and ion motions. Global well-posedness for the axisymmetric incompressible viscous and resistive Hall-MHD equations was established by Fan et al. \cite{fan2013well}, and later Li and Liu established a global regularity for the  viscous and resistive Hall-MHD equations with low regularity axisymmetric data \cite{i2023global}. Recently, kinds of ill-posedness results of Hall- and electron-MHD systems were shown in \cite{jeong2022cauchy} and Li provided the local well-posedness of 3D ideal Hall-MHD system with an azimuthal magnetic field \cite{LZ2024B}.

In this paper, the coefficients $\mu_0$, $\nu_0$ do not play an essential role in the proof, so without loss of generality, we set $\mu_0=\nu_0=1$. We consider the Hall-MHD with the following cylindrical coordinates: 
\begin{equation*}
	r=\sqrt{x_1^2+x_2^2},\ \theta=\arctan\frac{x_2}{x_1},\ z=x_3.
\end{equation*}
Assume both $\boldsymbol{u}$ and $\boldsymbol{h}$ are axially symmetric, which is
\begin{equation*}
	\left\{\begin{aligned}
		\boldsymbol{u}&=u_r(t,r,z)\boldsymbol{e_r}+u_z(t,r,z)\boldsymbol{e_z},\\
		\boldsymbol{h}&=h_{\theta}\boldsymbol{e_{\theta}},
	\end{aligned}\right.
\end{equation*}
where the basis vectors $\boldsymbol{e_r},\boldsymbol{e_{\theta}},\boldsymbol{e_z}$ are
	\begin{equation*}
		\boldsymbol{e_r}=\left(\frac{x_1}{r},\frac{x_2}{r},0\right),\ \boldsymbol{e_{\theta}}=\left( -\frac{x_2}{r},\frac{x_1}{r},0 \right), \ \boldsymbol{e_z}=(0,0,1).
	\end{equation*}
	
From the local well-posedness result \cite{LZ2024B}, it is clear that if the initial velocity $\boldsymbol{u_0} \cdot \boldsymbol{e_\theta}$ vanishes, then $u_\theta$ will vanish for all time. That is, by assuming $\boldsymbol{u_0}\cdot \boldsymbol{e_{\theta}}=h_r =h_z\equiv0$, and combining with the following result of $\nabla\times[(\nabla \times \boldsymbol{f})\times \boldsymbol{f}]$ in cylindrical coordinate: 
\begin{equation*}
	\begin{aligned}
		\nabla\times[(\nabla\times \boldsymbol{f})\times \boldsymbol{f}] = &\partial_z (j_zf_r -j_r f_z)\boldsymbol{e_r} + \left(\partial_z(j_{\theta}f_z - j_z f_{\theta})-\partial_r (j_rf_{\theta}-j_{\theta}f_r)\right)\boldsymbol{e_\theta}\\
		&\, + \frac{1}{r}\partial_r \left(r(j_zf_r-j_rf_z)\right)\boldsymbol{e_z},
	\end{aligned}
\end{equation*}
where
\begin{equation*}
	\left\{ \begin{aligned}
		&\boldsymbol{f}=f_r(t,r,z)\boldsymbol{e_r}+f_{\theta}(t,r,z)\boldsymbol{e_\theta}+f_z(t,r,z)\boldsymbol{e_z}, \\
		&\boldsymbol{j}= \nabla\times \boldsymbol{f}=j_r(t,r,z)\boldsymbol{e_r}+j_{\theta}(t,r,z)\boldsymbol{e_\theta}+j_z(t,r,z)\boldsymbol{e_z},
	\end{aligned}\right.
\end{equation*}
and 
\begin{equation*}
	j_r = -\partial_z f_{\theta},\qquad j_{\theta}=\partial_zf_r-\partial_rf_z,\qquad j_z=\frac{1}{r}\partial_r(rf_{\theta}),
\end{equation*}
one can easily rewrite the system \eqref{eq1.1} in the following way: 
\begin{equation}\label{A-N-MHD}
	\left\{ \begin{aligned}
		&\partial_t u_r + (u_r\partial_r + u_r\partial_z)u_r + \partial_r p = \mu(\Delta -\frac{1}{r^2})u_r - \frac{(h_{\theta})^2}{r},\\
		&\partial_t u_z + (u_r\partial_r + u_z \partial_z )u_z + \partial_z p =\mu\Delta u_z,\\
		&\partial_t h_{\theta} + (u_r\partial_r + u_z \partial_z )h_{\theta} - \frac{h_{\theta}u_r}{r} =\frac{\partial_z (h_{\theta})^2}{r},\\
		&\nabla\cdot \boldsymbol{u}= \partial_r u_r + \frac{u_r}{r} +\partial_z u_z =0.
	\end{aligned}\right.
\end{equation}
	
In this case, the vorticity $\boldsymbol{w}$ of the axially symmetric vector $\boldsymbol{u}$ is given by
\begin{equation*}
	\boldsymbol{w}=\mathrm{curl}\, \boldsymbol{u} = w_{\theta}\boldsymbol{e_{\theta}},
\end{equation*}
where 
\begin{equation*}
	w_{\theta}=\partial_z u_r - \partial_r u_z.
\end{equation*}
	
From \eqref{A-N-MHD} and the initial data \eqref{ini-data}, one can deduce that $\boldsymbol{w}$ satisfies: 
\begin{equation}\label{w-theta}
	\partial_t w_{\theta} + (u_r\partial_r+u_z\partial_z)w_{\theta} = \frac{u_r}{r}w_{\theta} - \frac{1}{r}\partial_z(h_{\theta})^2 + \mu(\Delta - \frac{1}{r^2})w_{\theta}.
\end{equation}
	
Recently, the Hall-MHD equations have received significant attention from mathematicians due to their crucial role in various plasma phenomena. Significant efforts have been dedicated to understanding their well-posedness and regularity. Chae, Degond and Liu in \cite{chae2014well} established the global existence of weak solutions and the local well-posedness for smooth solutions in Sobolev space $H^s(\mathbb{R}^3)$ with $s>5/2$. Subsequently, Dai \cite{dai2020local} extended the local well-posedness theory to $n$-dimensional case, in the space $H^s(\mathbb{R}^n)$ with $s>n/2$, for $n\geq 2$. For the non-resistive Hall-MHD system, Chae and Weng in \cite{chae2016singularity} showed it is not globally well-posed in any Sobolev space $H^s(\mathbb{R}^3)$ with $s>7/2$, even for smooth initial data. More recently, Li and Yang in \cite{LZYM2022} provided a single-component regularity criterion for the non-resistive axially symmetric Hall-MHD system, demonstrating that the regularity is correlated with a Beale-Kato-Majda (BKM) type condition on the reformulated magnetic quantity $\mathcal{H}: = \frac{h_{\theta}}{r}$.

Previous studies have established fundamental and significant insights into both resistive and non-resistive Hall-MHD equations, establishing criteria for either global regularity or finite-time blowup. However, quantitative estimates for the lifespan of strong solutions remain a relatively underdeveloped area. In the past few years, research related to the lifespan for PDEs of fluid dynamics has attracted considerable attention. Recently, Li-Zhou \cite{LZZT2024} gave an exact lifespan lower bound for axisymmetric incompressible Euler equations with a small swirl. For more related topics, see \cite{chemin2001lifespan,LP2002,HTLY2023,LZ2023,bae2025long} and references therein.
	
Our proof of main theorems is carried out mainly in the following two quantities:
\begin{align*}
		\Omega: =\frac{w_\theta}{r}, \quad \mathcal{H}: =\frac{h_\theta}{r}.
\end{align*}
Using $\eqref{A-N-MHD}_{1,2,3}$, one finds $(\Omega,\mathcal{H})$-system satisfies
\begin{equation}\label{OM}
	\left\{\begin{aligned}
		&\partial_t \Omega + \boldsymbol{u} \cdot \nabla \Omega = \mu\left(\Delta+\frac{2}{r} \partial_r\right) \Omega-\partial_z \mathcal{H}^2 , \\
		&\partial_t \mathcal{H}+\left(u_r \partial_r+u_z \partial_z\right) \mathcal{H}-2\mathcal{H} \partial_z \mathcal{H}=0.
	\end{aligned}\right.
\end{equation}

\subsection{Notations}
For the derivation that follows, we list the notations that will be used throughout the paper.
\begin{itemize}
	\item We use standard notations for Lebesgue and Sobolev functional spaces in $\mathbb{R}^3$: for $1 \leq p \leq \infty$ and $k \in \mathbb{N},$ $L^p$ denotes the Lebesgue space with norm
	\begin{equation*}
		\|f\|_{L^p}: = \begin{cases}\left(\int_{\mathbb{R}^3}|f(x)|^p \mathrm{~d} x\right)^{1 / p} & 1 \leq p<\infty, \\
			\underset{x \in \mathbb{R}^3}{\operatorname{ess} \sup }|f(x)| & p=\infty. \end{cases}
	\end{equation*}
	\item We write $A \lesssim B$ to indicate that $A \leq C B$ for some constant $C > 0$. Similarly, $A \simeq B$ means that both $A \lesssim B$ and $B \lesssim A$.
	\item The commutator of two operators $\mathcal{A}$ and $\mathcal{B}$ is defined as $[\mathcal{A}, \mathcal{B}] = \mathcal{A} \mathcal{B} - \mathcal{B} \mathcal{A}$.
	\item The notation $C_{a, b, \ldots}$ represents a positive constant that depends on the parameters $a, b, \ldots$, its value may vary from one occurrence to another.
	\item Let $\mathfrak{M} = (m_1, m_2, m_3)$ denote a multi-index with $m_1, m_2, m_3 \in \mathbb{N} \cup \{0\}$ and $|\mathfrak{M}| = m_1 + m_2 + m_3$. We define the higher-order derivative operator $\nabla^{\mathfrak{M}} = \partial_{x_1}^{m_1} \partial_{x_2}^{m_2} \partial_{x_3}^{m_3}$.
	\item $W^{k, p}$ denotes the usual Sobolev space with its norm
	\begin{align*}
		\|f\|_{W^{k, p}}: =\sum_{0 \leq|\mathfrak{M}| \leq k}\left\|\nabla^{\mathfrak{M}} f\right\|_{L^p}.
	\end{align*}
	\end{itemize}
We also simply denote $H^k$ and $\dot{H}^k$ instead of $W^{k, p}$ and $\dot{W}^{k, p}$ provided $p=2$.

\subsection{Main results}
Now we are ready to show the main results of the current paper. First, we consider the viscous case. We set the viscous coefficient $\mu=1$ without loss of generality. The results are as follows: 
\begin{theorem}[viscid case]\label{tvis}
	Let $(\boldsymbol{u}, \boldsymbol{h}) $ be a local smooth axially symmetric solution of system \eqref{eq1.1} with the initial data  $(\boldsymbol{u_0},\boldsymbol{h_0},\mathcal{H}_0)\in H^3(\mathbb{R}^3)$, $\boldsymbol{h}=h_{\theta}(t,r,z)\boldsymbol{e_{\theta}}$ satisfied $\nabla\cdot \boldsymbol{u_0} = \boldsymbol{u_0}\cdot \boldsymbol{e_{\theta}}=h_{r}=h_{z} \equiv 0$. Suppose
	\begin{equation*}
		\|\nabla \mathcal{H}_0\|_{L^{\infty}} = \varepsilon \ll 1,
	\end{equation*}
	then the solution $(\boldsymbol{u},\boldsymbol{h})$ keeps in $H^3$ when $t\leq T_*$, where $T_*$ satisfies: 
	\begin{equation}
		T_*\geq\frac{C_*}{(1+E_0)^\frac{4}{5}}\log(\log(\varepsilon^{-1}))^{4/5},
	\end{equation}
    here $C_*$ is a constant.
\end{theorem}

\qed

\begin{remark}
Roughly speaking, Theorem \ref{tvis} indicates the lifespan of system \eqref{A-N-MHD} can be arbitrarily large if the initial quantity $\mathcal{H}_0$ is close enough to a constant.
\end{remark}

\qed

Furthermore, we consider axially symmetric solutions of system \eqref{eq1.1} without the viscous term ($\mu=0$). For convenience, we rewrite the system \eqref{A-N-MHD} as follows:
\begin{align}\label{eq1.2}
	\left\{\begin{array}{l}
		\partial_t u_r+\left(u_r \partial_r+u_z \partial_z\right) u_r+\partial_r p=-\frac{\left(h_\theta\right)^2}{r}, \\
		\partial_t u_z+\left(u_r \partial_r+u_z \partial_z\right) u_z+\partial_z p=0, \\			\partial_t h_\theta+\left(u_r \partial_r+u_z \partial_z\right) h_\theta-\frac{h_\theta u_r}{r}=\frac{\partial_z\left(h_\theta\right)^2}{r}, \\
		\nabla \cdot \boldsymbol{u}=\partial_r u_r+\frac{u_r}{r}+\partial_z u_z=0.
	\end{array}\right.
\end{align}
	
By the first three equations of \eqref{eq1.2}, one deduces $w_\theta=\p_zu_r-\p_ru_z$ satisfies
\begin{align*}
	\partial_t w_\theta+\left(u_r \partial_r+u_z \partial_z\right) w_\theta=\frac{u_r}{r} w_\theta-\frac{1}{r} \partial_z\left(h_\theta\right)^2,
\end{align*}
and the inviscid reformulated $(\Omega, \mathcal{H})$-system follows 
\[ 
	\begin{aligned}
		\left\{\begin{array}{l}
			\partial_t \Omega+\boldsymbol{u} \cdot \nabla \Omega=-\partial_z \mathcal{H}^2, \\
			\partial_t \mathcal{H}+\left(u_r \partial_r+u_z \partial_z\right) \mathcal{H}-2 \mathcal{H} \partial_z \mathcal{H}=0.
		\end{array}\right.
	\end{aligned}
\]
    
Compared to viscid system, the lower bound on lifespan of inviscid system has more logarithmic factor. The result is given in the following: 
	
\begin{theorem}[inviscid case]\label{tivis}
	Let $(\boldsymbol{u}, \boldsymbol{h}) $ be a local smooth axially symmetric solution of \eqref{eq1.1} ($\mu=0$) with the initial data $\left(\boldsymbol{u_0}, \boldsymbol{h_0}, \mathcal{H}_0 \right) \in H^3\left(\mathbb{R}^3\right)$, $\boldsymbol{h}=h_{\theta}(t,r,z)\boldsymbol{e_{\theta}}$ satisfied $\nabla\cdot \boldsymbol{u_0} = \boldsymbol{u_0}\cdot \boldsymbol{e_{\theta}}=h_{r}=h_{z} \equiv 0$. Suppose
	\begin{equation}\label{CD1}
\|\nabla \boldsymbol{h}_0\|_{L^{\infty}}+\left\|\nabla\mathcal{H}_0\right\|_{L^{\infty}}=\ve \ll 1 ,
	\end{equation}
then the solution $(\boldsymbol{u}, \boldsymbol{h})(t, \cdot)$ keeps in $H^3$ when $t \leq T_*$, where $T_*$ satisfies: 
	\begin{align*}
		T_*=\frac{C_*}{1+E_0} \log\left(\log \left(\log \left(\log \left(\ve^{-1}\right)\right)\right)\right).
	\end{align*}
here $C_*$ is a constant.
\end{theorem}
	
\qed
	
\begin{remark}\label{rmk1}
We give an exact nontrivial example of the initial magnetic field that satisfies \eqref{CD1} since the condition looks strict.  It is inspired by the design of tokamaks \cite{tokamaks}. For $a,\,M>0$, let $\bm{h}_0=h_{\theta,0}\bm{e_\theta}$, where
\[
    h_{\theta,0}=M\phi\left(\frac{2r}{a}-3\right).
\]
Here $\phi$ is the standard smooth bump function:
\begin{align*}
    \phi(x)= \begin{cases}\exp \left(-\frac{1}{1-x^2}\right) & \text { if }\quad|x|<1, \\ 0 & \text { if }\quad |x| \geq 1 , \end{cases}
    \end{align*}
and thus $h_{\theta,0}$ is supported on $r\in[a,2a]$. Direct calculation shows
\[
    \begin{aligned}
        \p_rh_{\theta,0}&=\frac{2M}{a}\phi'\left(\frac{2r}{a}-3\right),\\
        \p_r\left(\frac{h_{\theta,0}}{r}\right)&=\frac{2M}{ar}\phi'\left(\frac{2r}{a}-3\right)-\frac{M}{r^2}\phi\left(\frac{2r}{a}-3\right).
    \end{aligned}
\]
This indicates
\[
\|\nabla \bm{h_0}\|_{L^\infty}\simeq \frac{M}{a},\quad \|\nabla\mathcal{H}_0\|_{L^\infty}\lesssim \frac{M}{a^2}.
\]
Thus by choosing $a=CM\ve^{-1}$, where $C>0$ is a universal constant, the condition \eqref{CD1} holds. 
    \end{remark}
    
\qed

\begin{remark}
By choosing $M>>1$, the size of magnetic field $\bm{h_0}$ constructed in Remark \ref{rmk1} can be arbitrarily large.
\end{remark}

\qed

The rest of this paper is organized as follows. To prepare to give a precise lower bound of lifespan of the non-resistive axially symmetric Hall-MHD system, we first compile some essential preliminaries. These include key lemmas on interpolation inequalities, commutator estimates, and the fundamental energy estimate for the system in Section \ref{Pre}. With these tools in hand, the main result of the viscid case is proved in Section \ref{sec3}, and the main result of inviscid case is proved in Section \ref{sec4}.
	
\section{Preliminaries}\label{Pre}
To begin with, we introduce some useful lemmas which will be frequently used in the proof of the main theorem.

We first present the Gagliardo-Nirenberg interpolation inequality, the proof of which is omitted here for brevity.
\begin{lemma}[Gagliardo-Nirenberg]{\label{G-N}}
	Fix $q, r \in[1, \infty]$ and $j, m \in \mathbb{N} \cup\{0\}$ with $j \leq m$. Suppose that $f\in L^q \cap \dot{W}^{m, r}$ and there exists a real number $\alpha \in[j / m, 1]$ such that
	\begin{align*}
		\frac{1}{p}=\frac{j}{3}+\alpha\left(\frac{1}{r}-\frac{m}{3}\right)+\frac{1-\alpha}{q}.
	\end{align*}
	Then $\boldsymbol{f} \in \dot{W}^{j, p}$ and there exists a constant $C>0$ such that
	\begin{align*}
		\left\|\nabla^j f\right\|_{L^p} \leq C\left\|\nabla^m f\right\|_{L^r}^\alpha\|f\|_{L^q}^{1-\alpha},
	\end{align*}
except the following two cases: 

(i) $j=0, m r<3$ and $q=\infty$; (In this case it is necessary to assume also that either $f \rightarrow 0$ at infinity, or $f \in L^s$ for some $s<\infty$.)
		
(ii) $1<r<\infty$ and $m-j-3 / r \in \mathbb{N}$. (In this case it is necessary to assume also that $\alpha<1$.)
\end{lemma}

\qed

Next, we present the following estimate for the triple product form that will be frequently used in the final proof.
\begin{lemma}{\label{Kato-Ponce}}
	Let $m \in \mathbb{N}$ and $m \geq 2, \boldsymbol{f}, \boldsymbol{g}, \boldsymbol{k} \in C_0^{\infty}\left(\mathbb{R}^3\right)$. The following estimates hold: 
	\begin{align}{\label{lemKT1}}
		\left|\int_{\mathbb{R}^3}\left[\nabla^m, \boldsymbol{f} \cdot \nabla\right] \boldsymbol{g} \nabla^m \boldsymbol{k} d x\right| \leq C\left\|\nabla^m(\boldsymbol{f}, \boldsymbol{g}, \boldsymbol{k})\right\|_{L^2}^2\|\nabla(\boldsymbol{f}, \boldsymbol{g})\|_{L^{\infty}} ;
	\end{align}
\end{lemma}

\begin{proof}
    We apply the Hölder inequality, one derives
	\begin{align}{\label{lemKT2}}
		\left|\int_{\mathbb{R}^3}\left[\nabla^m, \boldsymbol{f} \cdot \nabla\right] \boldsymbol{g} \nabla^m \boldsymbol{k} d x\right| \leq\left\|\left[\nabla^m, \boldsymbol{f} \cdot \nabla\right] \boldsymbol{g}\right\|_{L^2}\left\|\nabla^m \boldsymbol{k}\right\|_{L^2}.
	\end{align}
	Due to the commutator estimate by Kato-Ponce \cite{KTPG2010}, it follows that
	\begin{align}{\label{lemKT3}}
		\left\|\left[\nabla^m, \boldsymbol{f} \cdot \nabla\right] \boldsymbol{g}\right\|_{L^2} \leq C\left(\|\nabla \boldsymbol{f}\|_{L^{\infty}}\left\|\nabla^m \boldsymbol{g}\right\|_{L^2}+\|\nabla \boldsymbol{g}\|_{L^{\infty}}\left\|\nabla^m \boldsymbol{f}\right\|_{L^2}\right).
		\end{align}
	Then \eqref{lemKT1} follows from substituting \eqref{lemKT3} in \eqref{lemKT2}.
\end{proof} 
    
The following is the fundamental energy estimate and the $L^p$ conservation of $\mathcal{H}$.
\begin{lemma}\label{L^pconsx}
	Define $\mathcal{H}: =\frac{h_\theta}{r}$. Let $(\boldsymbol{u}, \boldsymbol{h}) \in H^3$ be the solution of \eqref{A-N-MHD}, then we have the following states in different case: 
    For $p \in[2, \infty]$ and $t \in(0, \infty)$, the viscid case ensures the following estimate holds: 
    \begin{equation}\label{viscid L^p con-H}
        \|(\boldsymbol{u},\boldsymbol{h})(t,\cdot)\|_{L^2}^2+2\int_0^t \|\nabla\boldsymbol{u}(s,\cdot)\|^2_{L^2}ds\leq \|\boldsymbol{u_0},\boldsymbol{h_0}\|_{L^2}^2.
    \end{equation}
In the inviscid case, \eqref{viscid L^p con-H} reduces to
    \begin{align}{\label{Fundamental eq1}}
	   \begin{gathered}
			\|(\boldsymbol{u}, \boldsymbol{h})(t, \cdot)\|_{L^2}^2 \leq\left\|\left(\boldsymbol{u_0}, \boldsymbol{h_0}\right)\right\|_{L^2}^2 .
		\end{gathered}
	\end{align}
Both viscid case and inviscid case hold
    \begin{equation}\label{L^p-con}
        \|\mathcal{H}(t, \cdot)\|_{L^p}=\left\|\mathcal{H}_0\right\|_{L^p} .
    \end{equation}
\end{lemma}

\begin{proof}
    Inequalities \eqref{viscid L^p con-H} and \eqref{Fundamental eq1} follow from the standard $L^2$ energy estimate of the system \eqref{A-N-MHD}, by $\eqref{A-N-MHD}_3$ in different cases, we have the same estimate of $L^p$ conservation of $\mathcal{H}$, we all omit the details here.
\end{proof}

Not influenced by the loss of the viscous term, the next three lemmas are devoted to some basic estimates of the system \eqref{eq1.1}. These lemmas will be used both in the proof of Theorem \ref{tvis} and Theorem \ref{tivis}.
\begin{lemma}\label{h_c w_c}
	Let $(\boldsymbol{u}, \boldsymbol{h}) $ be a local smooth axially symmetric solution of the system \eqref{eq1.1} on $t\in[0,T)$ with the initial data  $(\boldsymbol{u_0},\boldsymbol{h_0},\mathcal{H}_0)\in H^3(\mathbb{R}^3)$, $\boldsymbol{h}=h_{\theta}(t,r,z)\boldsymbol{e_{\theta}}$ satisfied $\nabla\cdot \boldsymbol{u_0} = \boldsymbol{u_0}\cdot \boldsymbol{e_{\theta}}=h_{r}=h_{z} \equiv 0$. One has the following $L^{\infty}_{t}L^p$ estimates of $h_{\theta}$ and $w_{\theta}$, for all $p \in (1, \infty]$:
	\begin{equation*}
		\begin{aligned}
			\|h_{\theta}(t,\cdot)\|_{L^p} &\leq \|\boldsymbol{h_0}\|_{L^p}\exp\left( C\int_0^t\left[ \|\frac{u_r}{r}(s,\cdot)\|_{L^{\infty}} + \|\partial_z \mathcal{H}(s,\cdot)\|_{L^{\infty}} \right]ds \right),\\
			\|w_{\theta}(t,\cdot)\|_{L^p} &\leq (\|\boldsymbol{w_0}\|_{L^p} + \left\|h_\theta\right\|_{L_{t}^{\infty} L^p}\left\|\partial_z \mathcal{H}\right\|_{L_{t}^1 L^{\infty}}) \\
			&\qquad \times\exp\left( \int_0^t \|\frac{u_r}{r}(s,\cdot)\|_{L^{\infty}}ds \right).
		\end{aligned}
	\end{equation*}
\end{lemma}

\qed
	
The next lemma states the $L^p(1<p<\infty)$ norm of $\nabla \frac{v_r}{r}$ could be controlled by the $L^p$ norm of $\frac{w_\theta}{r}$, whose proof could be found in (\cite{MCZX2013}, Proposition 2.5) and (\cite{LZ2015} equation A.5).

\begin{lemma}{\label{nav_r cl}}
	Define $\Omega: =\frac{w_\theta}{r}$. For $1<p<\infty$, there exists an absolute constant $C_p>0$ such that
	\begin{align*}
		\left\|\nabla \frac{u_r}{r}(t, \cdot)\right\|_{L^p} \leq C_p\|\Omega(t, \cdot)\|_{L^p} .
	\end{align*}
\end{lemma}

\qed
    
Next, we present several relevant estimates, which will be employed both in the proof of Theorem \ref{tvis} and in the proof of Theorem \ref{tivis}. 
\begin{lemma}\label{nahH eges}
    Under the same assumptions as Lemma \ref{h_c w_c}, one has the following estimates of $\|\boldsymbol{u},\boldsymbol{h}, \mathcal{H}\|_{H^3}^2$, $\|\nabla \boldsymbol{h}\|_{L^\infty}$ and $\|\nabla\mathcal{H}\|_{L^\infty}$: 
    \begin{flalign}
        \label{nabla h}
        &\qquad \|\nabla \boldsymbol{h}(t,\cdot)\|_{L^\infty}\leq \|\nabla \boldsymbol{h_0}\|_{L^\infty}\exp\left( C\int_0^t (\|\nabla \boldsymbol{u}(s,\cdot)\|_{L^\infty}+\|\partial_z \mathcal{H}(s,\cdot)\|_{L^{\infty}})ds \right), & \\
        \label{nabla H}
        &\qquad \|\nabla \mathcal{H}(t,\cdot)\|_{L^{\infty}} \leq \|\nabla \mathcal{H}_0\|_{L^{\infty}}\exp\left( C\int_0^t (\|\nabla \boldsymbol{u}(s,\cdot)\|_{L^{\infty}} +\|\partial_z \mathcal{H} (s,\cdot)\|_{L^{\infty}})ds \right), & \\
        \label{EsOfuhH}
        &\qquad \|(\boldsymbol{u},\boldsymbol{h},\mathcal{H})(t,\cdot)\|_{H^3}^2 \leq 	\|\boldsymbol{u_0},\boldsymbol{h_0},\mathcal{H}_0\|_{H^3}^2 \exp\left( C \int_0^t\|\nabla(\boldsymbol{u},\boldsymbol{h},\mathcal{H})(t, \cd)\|_{L^{\infty}}ds\right). & 
    \end{flalign}
\end{lemma}

\begin{proof}
    At the beginning, we notice that
        \begin{align*}
    	   |\nabla \boldsymbol{h}| \simeq\left|\partial_r h_\theta\right|+\left|\partial_z h_\theta\right|+|\mathcal{H}| .
        \end{align*}
    Since the $L_{t}^{\infty} L^{\infty}$ -estimate of $\mathcal{H}$ has already derived in Lemma \ref{L^pconsx}, one only focuses on the rest two above. Acting $\bar{\nabla}=\left(\partial_r, \partial_z\right)$ on $\eqref{A-N-MHD}_3$, respectively, and performing $L^p$ $(2 \leq p<\infty)$ energy estimates on each resulting equation, by integration by parts and Hölder's inequality, one deduces
    \begin{align}{\label{Es nah}}
    	\begin{gathered}
    		\frac{d}{d t}\left\|\bar{\nabla} h_\theta(t, \cdot)\right\|_{L^p}^p \leq C p\left\|\bar{\nabla} h_\theta(t, \cdot)\right\|_{L^p}^{p-1}\left(\|\nabla \boldsymbol{u}(t, \cdot)\|_{L^{\infty}}+\left\|\partial_z \mathcal{H}(t, \cdot)\right\|_{L^{\infty}}\right) \\
    		\times\left(\left\|\bar{\nabla} h_\theta(t, \cdot)\right\|_{L^p}+\|\mathcal{H}(t, \cdot)\|_{L^p}\right).
    	\end{gathered}
    \end{align}
    Canceling $p\left\|\bar{\nabla} h_\theta(t, \cdot)\right\|_{L^p}^{p-1}$ on each sides of \eqref{Es nah}, noting $\frac{d}{d t}\|\mathcal{H}(t, \cdot)\|_{L^p} \equiv 0$ from Lemma \ref{L^pconsx}, and applying one Grönwall's inequality, one arrives at 
    \begin{align*}
    	   \|\nabla \boldsymbol{h}(t, \cdot)\|_{L^p} \leq\left\|\nabla \boldsymbol{h_0}\right\|_{L^p} \exp \left(C \int_0^t\left(\|\nabla \boldsymbol{u}(s, \cdot)\|_{L^{\infty}}+\left\|\partial_z \mathcal{H}(s, \cdot)\right\|_{L^{\infty}}\right) d s\right).
    \end{align*}
    Note that the constant $C$ above is independent with $p \in[2, \infty)$. Let $p \rightarrow \infty$, one concludes the estimate \eqref{nabla h}.
        
    Next, we present a concise derivation of the estimate \eqref{nabla H}.  Acting $\partial_r$ on both sides of equation
    \begin{equation*}
          \partial_z\mathcal{H}+(u_r\partial_r+u_z\partial_z)\mathcal{H} - 2\mathcal{H}\partial_z\mathcal{H}=0, 
    \end{equation*}
    followed by multiplication with $p\partial_r\mathcal{H}|\mathcal{H}|^{p-2}$, and subsequent integration over $\mathbb{R}^3$ combining with integration by parts, one derives
    \begin{equation}
        \frac{d}{dt}\|\partial_r\mathcal{H}(t,\cdot)\|^p_{L^p} \leq 2(p-1)\int_{\mathbb{R}^3}\partial_z\mathcal{H}|\partial_r\mathcal{H}|^pdx + Cp\int_{\mathbb{R}^3}|\nabla \boldsymbol{u}||\nabla \mathcal{H}| |\partial_r \mathcal{H}|^{p-1}dx.
    \end{equation}
    Applying Hölder's inequality to the above equality, one arrives
    \begin{equation}\label{rHe}
        \frac{d}{dt}\|\partial_r\mathcal{H}(t,\cdot)\|_{L^p}^p \lesssim p(\|\nabla \boldsymbol{u}(t,\cdot)\|_{L^{\infty}} + \|\partial_z\mathcal{H}(t,\cdot)\|_{L^{\infty}})\|\nabla\mathcal{H}(t,\cdot)\|^p_{L^p}.
    \end{equation}
    Similarly, one has
    \begin{equation}\label{zHe}
        \frac{d}{dt}\|\partial_z\mathcal{H}(t,\cdot)\|_{L^p}^p \lesssim p(\|\nabla \boldsymbol{u}(t,\cdot)\|_{L^{\infty}} + \|\partial_z\mathcal{H}(t,\cdot)\|_{L^{\infty}})\|\nabla\mathcal{H}(t,\cdot)\|^p_{L^p}.
    \end{equation}
    Combing \eqref{rHe} and \eqref{zHe}, canceling $p\|\nabla\mathcal{H}(t,\cdot)\|^{p-1}_{L^p}$ on each side of the inequality, one concludes the estimate \eqref{nabla H}.

    Finally, we give a concise proof of the higher-order bounds of \eqref{A-N-MHD} and \eqref{eq1.2}. The key of the proof is performing energy estimates of the system together with the equation of $\mathcal{H}$, in the viscid case, it means
    \begin{align}{\label{invieq}}
        \left\{\begin{array}{l}
            \partial_t \boldsymbol{u}+\boldsymbol{u} \cdot \nabla \boldsymbol{u}+\nabla p-\Delta \boldsymbol{u}=\boldsymbol{h} \cdot \nabla \boldsymbol{h}, \\
    	   \partial_t \boldsymbol{h}+\boldsymbol{u} \cdot \nabla \boldsymbol{h}-\boldsymbol{h} \cdot \nabla \boldsymbol{u}=2 \mathcal{H} \partial_z \boldsymbol{h}, \\
    	   \partial_t \mathcal{H}+\boldsymbol{u} \cdot \nabla \mathcal{H}-2 \mathcal{H} \partial_z \mathcal{H}=0 .
        \end{array}\right.
    \end{align}

    Applying $\nabla^3$ to three equations in \eqref{invieq}, and performing the $L^2$ inner product of the resulting equations with $\nabla^3 \boldsymbol{u}$, $\nabla^3 \boldsymbol{h}$ and $\nabla^3 \mathcal{H}$ respectively, noting that $\boldsymbol{h}$ is divergence-free, and apply Lemma \refeq{Kato-Ponce} to the remaining terms of the functions, we can finally arrive 
    \begin{align*}
        \frac{d}{d t}\|(\boldsymbol{u}, \boldsymbol{h}, \mathcal{H})(t, \cdot)\|_{\dot{H}^3}^2 \leq C\|\nabla(\boldsymbol{u}, \boldsymbol{h}, \mathcal{H})(t, \cdot)\|_{L^{\infty}}\|(\boldsymbol{u}, \boldsymbol{h}, \mathcal{H})(t, \cdot)\|_{\dot{H}^3}^2. 
    \end{align*}
    Combining this with the fundamental energy estimate $\eqref{viscid L^p con-H}$ and the $L^p$ conservation law of $\mathcal{H}$  $\eqref{L^p-con}$, thus the energy estimate of the system is proved by applying Grönwall's inequality, which indicates
    \begin{align*}
        \|(\boldsymbol{u}, \boldsymbol{h}, \mathcal{H})(t, \cdot)\|_{H^3}^2 \leq\left\|\left(\boldsymbol{u_0}, \boldsymbol{h_0}, \mathcal{H}_0\right)\right\|_{H^3}^2 \exp \left(C \int_0^t\|\nabla(\boldsymbol{u}, \boldsymbol{h}, \mathcal{H})(s, \cdot)\|_{L^{\infty}} d s\right).
    \end{align*}
    In inviscid case, the only difference is that the missing of $\Delta u$ brings is we lose a positive term $\left\|\nabla^{4} u(t, \cdot)\right\|_{L^2}^2$ on the left hand side of the energy estimate equation, which makes no influence to the result of the higher-order bound estimate. More details of the proof refers to \cite{LZYM2022}.
\end{proof}

\section{Proof of Theorem \ref{tvis}}\label{sec3}
This section is devoted to the proof of Theorem \ref{tvis}.  Firstly, we establish the $L_{t}^{\infty}L^{\infty}$ -boundedness of the quantity $\frac{u_r}{r}$. Secondly, we derive an $L^{\infty}$ estimate for $\nabla\boldsymbol{u}$. Having completed these preliminary analyses, we proceed to present the proof of Theorem \ref{tvis}, which constitutes the final component of our argument.

\subsection{$L_t^{\infty}L^{\infty}$ -Boundedness of $\frac{u_r}{r}$} \label{sec3.1}
We first present the estimate for $\frac{u_r}{r}$ in the form of a proposition and provide a brief proof.
    
\begin{proposition}\label{pro-u_r/r}
    Define $\Omega: =\frac{w_{\theta}}{r}$. Assume that $\nabla\cdot\boldsymbol{u_0}=\boldsymbol{u_0}\cdot\boldsymbol{e_{\theta}}=h_r=h_z\equiv0$. Let $(\boldsymbol{u},\boldsymbol{h})$ be the unique local axially symmetric solution of \eqref{eq1.1} on $t\in[0,T)$ with the initial data $(\boldsymbol{u_0},\boldsymbol{h_0})\in H^m(m\geq 3)$. Then the following estimate of $\frac{u_r}{r}$ holds uniformly: 
    \begin{equation}\label{Es-u_r-r}
        \int_0^t \|\frac{u_r}{r}(s,\cdot)\|_{L^{\infty}}ds\lesssim t^{3/4}(\|\Omega_0\|_{L^2}+t^{1/2}\|\mathcal{H}_0\|_{L^4}^2).
    \end{equation}
\end{proposition}

\begin{proof}
    Applying the $L^2$ energy estimate on the equation of $\Omega=\frac{w_{\theta}}{r}$: 
    \begin{equation}\label{Eq-Omega}
        \partial_r\Omega+\boldsymbol{u}\cdot\nabla\Omega= (\Delta + \frac{2}{r}\partial_r)\Omega-\partial_z\mathcal{H}^2,
    \end{equation}
    which can be derived from $\eqref{w-theta}_1$ by direct calculations. Performing the $L^2$ estimates for \eqref{Eq-Omega}, one arrives
    \begin{equation}\label{3.2L2}
        \frac{1}{2}\frac{d}{dt}\|\Omega(t,\cdot)\|^2_{L^2} + \|\nabla \Omega(t,\cdot)\|_{L^2}^2 = -\frac{1}{2}\underbrace{\int_{\mathbb{R}^3}\boldsymbol{u}\cdot\nabla\Omega^2dx}_{I_1} + \underbrace{\int_{\mathbb{R}^3} \frac{\partial_r\Omega^2}{r}dx}_{I_2} - \underbrace{\int_{\mathbb{R}^3} \Omega \partial_z \mathcal{H}^2dx}_{I_3}.
    \end{equation}
    By the divergence-free property of $\boldsymbol{u}$ and using integration by parts, one can derive the vanishing of $I_1$. Writing in the cylindrical coordinates, one can deduce that $I_2$ satisfies
    \begin{equation*}
        I_2=2\pi \int_{\mathbb{R}}\int_0^{\infty}\partial_r\Omega^2drdz=-2\pi\int_{\mathbb{R}}|\Omega(t,0,z)|^2dz\leq 0.
    \end{equation*}
    Using integration by parts together with Hölder's inequality and Young's inequality, one arrives
    \begin{equation*}
    \begin{aligned}
        |I_3|=\left|\int_{\mathbb{R}^3}\partial_z\Omega\mathcal{H}^2dx\right| &\leq \|\partial_z\Omega(t,\cdot)\|_{L^2}\|\mathcal{H}(t,\cdot)\|^2_{L^4}\leq\frac{1}{2}\|\nabla\Omega(t,\cdot)\|^2_{L^2}+\frac{1}{2}\|\mathcal{H}(t,\cdot)\|^4_{L^4}.
    \end{aligned}
    \end{equation*}
    Substituting above estimates for $I_1-I_3$ in \eqref{3.2L2}, and integrating over $(0,t)$, one can deduce
    \begin{equation}\label{pre-eO}
        \|\Omega(t,\cdot)\|_{L^2}^2 + \int_0^t\|\nabla\Omega(s,\cdot)\|^2_{L^2} ds \leq \|\Omega_0\|^2_{L^2} + \int_{0}^t\|\mathcal{H}(s,\cdot)\|_{L^4}^4ds.
    \end{equation}
    Finally, Using Gagliardo-Nirenberg inequality, together with Sobolev inequality and \eqref{nav_r cl}, one can deduce
    \begin{equation}\label{reul-g}
        \|\frac{u_r}{r}(t,\cdot)\|_{L^\infty}\leq C\|\nabla\frac{u_r}{r}(t,\cdot)\|^{1/2}_{L^2}\|\nabla\frac{u_r}{r}(t,\cdot)\|^{1/2}_{L^6}\leq C\|\Omega(t,\cdot)\|^{1/2}_{L^2}\|\nabla\Omega(t,\cdot)\|^{1/2}_{L^2}.
    \end{equation}
    Then by integrating on $(0,t)$ and recalling \eqref{pre-eO}, one concludes the estimate \eqref{Es-u_r-r}.
    \end{proof}
    
\subsection{$L^\infty$ Estimate of $\nabla\boldsymbol{u}$} \label{sec 3.2}
In this subsection we present an estimate for $\nabla \boldsymbol{u}$ by applying the maximal regularity of heat flows. Before stating this result, we first introduce the following lemma, which states the standard maximal regularity of heat flows in $L^q_TL^p$ type spaces. Next, we present the Biot-Savart law.
    
\begin{lemma}[Maximal $L^q_TL^p$ -regularity for the heat flow]\label{maximal re}
    Let us define the operator $\mathcal{A}$ by the formula
    \begin{equation*}
        \mathcal{A}: \qquad f\longmapsto \int_0^t \nabla^2e^{(t-s)\Delta}f(s,\cdot)ds.
    \end{equation*}
    Then $\mathcal{A}$ is bounded from $L^q(0,T;L^p(\mathbb{R}^d))$ to itself every $T\in(0,\infty]$ and $1<p$, $q<\infty$. Moreover, there holds: 
    \begin{equation*}
        \|\mathcal{{A}}f\|_{L^q(0,T;L^p(\mathbb{R}^d))} \leq C\|f\|_{L^q(0,T;L^{p}(\mathbb{R}^d))}.
    \end{equation*}
\end{lemma}
    
\qed
    

\qed
    
\begin{lemma}\label{lemma3.6}
    Under the system \eqref{eq1.1}, by using Lemma \ref{maximal re}, one has the following $L^{2}_{t}L^4$ estimate of $\nabla \boldsymbol{w}$: 
    \begin{equation*}
    	\|\nabla \boldsymbol{w}\|_{L_t^2L^4} \leq C t^{1/2}\left( \|w_{\theta}\|_{L_t^{\infty}L^4}\|w_{\theta}\|_{L^{\infty}(L^2\cap L^6)} + \|\mathcal{H}\|_{L_t^{\infty}L^{\infty}}\|h_{\theta}\|_{L_t^{\infty}L^4} \right).
    \end{equation*}
Here $C>0$ is a universal constant. 
\end{lemma}

 \begin{proof}
 Recalling $\boldsymbol{w}=\mathrm{curl}\,\boldsymbol{u}$ and rewriting the equation $\eqref{w-theta}_1$ in the vector form, one deduces
    \begin{equation}\label{heatflow}
        \partial_t\boldsymbol{w}-\Delta\boldsymbol{w}=-\nabla\times(\boldsymbol{w}\times\boldsymbol{u})-\partial_z(\mathcal{H}\boldsymbol{h}).
    \end{equation}
Similarly as \eqref{reul-g} and combining the Biot-Savart law in Lemma \ref{nav_r cl}, one deduces
    \begin{equation}\label{uleqw}
        \|\boldsymbol{u}(t,\cdot)\|_{L^\infty}\leq C\|\nabla \boldsymbol{u}(t,\cdot)\|^{1/2}_{L^2}\|\nabla\boldsymbol{u}(t,\cdot)\|^{1/2}_{L^6}\leq C\|\boldsymbol{w}(t,\cdot)\|_{(L^2\cap L^6)}.
    \end{equation}
    By Hölder's inequality and \eqref{uleqw}, one can deduce
    \begin{equation*}
       \| (\boldsymbol{w}\times \boldsymbol{u})(t,\cdot) \|_{L^4}\leq C\|w_{\theta}(t,\cdot)\|_{L^4}\|w_{\theta}(t,\cdot)\|_{(L^2\cap L^6)},
    \end{equation*}
thus one has
    \begin{equation*}
       \| (\boldsymbol{w}\times \boldsymbol{u}) \|_{L_{t}^{\infty}L^4}\leq C\|w_{\theta}\|_{L_{t}^{\infty}L^4}\|w_{\theta}\|_{L_{t}^{\infty}(L^2\cap L^6)}.
    \end{equation*}
Similarly, using Hölder's inequality, Lemma \ref{L^pconsx} and Lemma \ref{h_c w_c}, one derives
    \begin{equation*}
        \|\mathcal{H}\boldsymbol{h}\|_{L^{\infty}_tL^4} \leq \|\mathcal{H}\|_{L^{\infty}_tL^4}\|h_{\theta}\|_{L^{\infty}_tL^4}.
    \end{equation*}
By employing the maximal regularity property of the heat flow described in equation \eqref{heatflow}, one can deduce
    \begin{equation*}
        \begin{aligned}
            \|\nabla\boldsymbol{w}\|_{L^2_tL^4}&\leq C\left(\|\boldsymbol{w}\times\boldsymbol{u}\|_{L^2_tL^4} + \|\mathcal{H}\boldsymbol{h}\|_{L^2_tL^4}\right)\\
            & \leq Ct^{1/2}\left(\|\omega_{\theta}\|_{L^{\infty}_tL^4}\|w_{\theta}\|_{L^{\infty}_t(L^2\cap L^6)} + \|\mathcal{H}\|_{L^{\infty}_tL^{\infty}}\|h_{\theta}\|_{L^{\infty}_tL^4}\right).
        \end{aligned}
    \end{equation*}
\end{proof}
    
Next, we state the estimate of $\nabla\boldsymbol{u}$ and present a concise demonstration.
\begin{corollary}
    Under the same assumptions as Proposition \ref{pro-u_r/r} one has the following estimate of $\nabla\boldsymbol{u}$:
    \begin{equation}\label{nabla u}
    	\int_0^t \|\nabla \boldsymbol{u}(s,\cdot)\|_{L^{\infty}} ds \lesssim t^{4/7}\|\boldsymbol{w}\|_{L^{\infty}_tL^2}^{1/7}\|\nabla \boldsymbol{w}\|_{L^2_tL^4}^{6/7}.
    \end{equation}
\end{corollary}

\begin{proof}
By invoking Lemma \ref{nav_r cl} in conjunction with the Gagliardo-Nirenberg inequality, this estimate can be rigorously substantiated.
\end{proof}

\subsection{End of the proof}\label{end of the1.1}
Since the non-resistive Hall-MHD system with an azimuthal magnetic field is locally well-posed in $H^3$ (see \cite{LZ2024B} for the main result) \footnote{Although the main result of \cite{LZ2024B} gives the local well-posedness for the 3D inviscid Hall-MHD system, it is sufficient to guarantee the local well-posedness for the viscid system via the same approach.}, there exists $T_*>0$ such that
\begin{equation}\label{preEs}
	t\sup_{0\leq s\leq t} \|\nabla\mathcal{H}(s,\cdot)\|_{L^{\infty}}\leq 1, \qquad \text{for all } t\in[0, T_*).
\end{equation}
In the following, our argument will be carried out before this $T_*$. Substituting the estimate \eqref{Es-u_r-r} into Lemma \ref{h_c w_c}, one deduces
\begin{equation*}
	\begin{aligned}
		\|h_{\theta}(t,\cdot)\|_{L^p}&\leq \|\boldsymbol{h_0}\|_{L^p}\exp\left\{ C\int_0^t\left[\|\frac{u_r}{r}(s,\cdot)\|_{L^{\infty}} + \| \partial_z \mathcal{H}(s,\cdot) \|_{L^{\infty}}\right]ds \right\} \\
		&\lesssim \|\boldsymbol{h_0}\|_{L^p}\exp\left\{ C\|\mathcal{H}_0\|_{L^4}^2t^{5/4} + C\|\Omega_0\|_{L^2}t^{3/4} + C\|\partial_z \mathcal{H}\|_{L_t^1L^{\infty}} \right\} \\
	\end{aligned}
\end{equation*}
for all $t\in[0, T_*)$. By integrating \eqref{preEs}, the above estimate can be further simplified to
\begin{equation}\label{E-h-Sim}
    \|h_{\theta}(t,\cdot)\|_{L^p} \leq E_0\exp\left( CE_0(t^{5/4}+t^{3/4})+1 \right),
\end{equation}
where $E_0 := \|\boldsymbol{u_0},\boldsymbol{h_0},\mathcal{H}_0\|_{H^3}$. Similarly as \eqref{E-h-Sim}, substituting the estimate \eqref{Es-u_r-r} into Lemma \ref{h_c w_c} and using \eqref{preEs}, one has
\begin{equation*}
    \|w_{\theta}(t,\cdot)\|_{L^p} \leq CE_0\exp\left( CE_0(t^{5/4}+t^{3/4})+1 \right).
\end{equation*}
	
Therefore, by the $L^p$ conservation of $\mathcal{H}$ and Corollary \ref{lemma3.6} , one can deduce $\| \nabla \boldsymbol{w} \|_{L_t^2L^4}$ satisfies
\begin{equation*}
	\begin{aligned}
		\|\nabla \boldsymbol{w}\|_{L^2_tL^4} &\leq Ct^{1/2} \left( \|w_{\theta}\|_{L^{\infty}_tL^4} \|w_{\theta}\|_{L^{\infty}_t(L^2\cap L^6)} + \|\mathcal{H}\|_{L^{\infty}_tL^{\infty}}\|h_{\theta}\|_{L^{\infty}_tL^4} \right)\\
		& \leq CE_0^2t^{1/2}\exp\left( CE_0(t^{5/4}+t^{3/4})+1 \right),
	\end{aligned}
\end{equation*}
that is to say the inequality \eqref{nabla u} can be written as
\begin{equation}\label{Enablau}
	\int_0^t \|\nabla \boldsymbol{u}(s,\cdot)\|_{L^{\infty}}ds \leq t^{4/7}\|\boldsymbol{w}\|^{1/7}_{L^{\infty}_tL^2}\|\nabla \boldsymbol{w}\|^{6/7}_{L^2_{t}L^4} \leq C(1+E_0)^2t\exp\left\{CE_0(t^{5/4}+t^{3/4})+1\right\}.
\end{equation}
Substituting the inequality (\ref{Enablau}) into (\ref{nabla h}) and (\ref{nabla H}), one derives: 
\begin{equation}\label{Enablah}
	\|\nabla \boldsymbol{h}(t,\cdot)\|_{L^{\infty}} \lesssim E_0\exp\left\{ C(1+E_0)^2t\exp\big(CE_0\big(t^{5/4}+t^{3/4})+1\big) +1 \right\},
\end{equation}
\begin{equation}\label{EnablaH}
	\|\nabla \mathcal{H}(t,\cdot)\|_{L^{\infty}}  \lesssim E_0\exp\left\{ C(1+E_0)^2t\exp\big(CE_0\big(t^{5/4}+t^{3/4})+1\big) +1 \right\}.
\end{equation}
By (\ref{EnablaH}), one deduces the following restriction: when $t=T_*$ ,
\begin{equation}
	\varepsilon  E_0T_*\exp\left\{ C(1+E_0)^2T_*\exp\big(CE_0\big(T_*^{5/4}+T_*^{3/4})+1\big) +1 \right\} \leq \frac{1}{2}
\end{equation}
ensures that the a priori assumption (\ref{preEs}) holds. Without loss of generality, we assume that $\varepsilon$ is sufficiently small, therefore $T_*>1$. By $T_*^{5/4}+T_*^{3/4} \leq CT_*^{5/4}$ for all $T_*>1$, one has
\begin{equation*}
\begin{aligned}
    \varepsilon  E_0T_*\exp\left\{ C(1+E_0)^2T_*\right.&\left.\exp\big(CE_0\big(T_*^{5/4}+T_*^{3/4})+1\big) +1 \right\} \\
    &\leq \varepsilon E_0T_* \exp\left\{ C(1+E_0)^2T_*\exp\left\{C(1+E_0)T_*^{5/4}\right\} \right\}.
\end{aligned}
\end{equation*}
A basic inequality $\log (1+t) \leq t \leq \mathrm{e}^t-1 $, for all $t\geq 0$, implies that
\begin{equation*}
	\begin{aligned}
		E_0T_*&\exp\left\{ C(1+E_0)^2T_*\exp\{C(1+E_0)T_*^{5/4}\} \right\} \\
        &= \exp\left\{ \log E_0T_* + C(1+E_0)^2T_*\exp\{C(1+E_0)T_*^{5/4}\} \right\}\\
        &\leq \frac{1}{2}\exp\{\exp(C(1+E_0)T_*^{5/4})\}.
	\end{aligned}
\end{equation*}
Thus,
\begin{equation*}
	\varepsilon E_0T_*\exp\left\{ C(1+E_0)T_*\exp\{C(1+E_0)T_*^{5/4}\} \right\} \leq \frac{\varepsilon}{2}\exp\{\exp(C(1+E_0)T_*^{5/4})\}.
\end{equation*}
Therefore, the condition (\ref{preEs}) is satisfied provided
\begin{equation*}
	\exp\{\exp(C_*^{-1}(1+E_0)T_*^{5/4})\} =\frac{1}{\varepsilon},
\end{equation*}
for sufficiently small $\varepsilon >0$ and some constant $C_*>0$, one finds
\begin{equation*}
	T_*=\frac{C_*}{(1+E_0)^\frac{4}{5}}\log(\log(\varepsilon^{-1}))^{4/5}.
\end{equation*}
This gives the desired lower bound of the lifespan in Theorem \ref{tvis}.
	
\section{Proof of Theorem \ref{tivis}}\label{sec4}
This section is devoted to the proof of Theorem \ref{tivis}. Firstly, we re-establish the $L_t^{\infty} L^p $ -estimate for $\Omega$ and the $L_t^\infty L^{\infty}$ -estimate for $w_\theta$. We then focus on the $L^\infty$ estimate for $\nabla u$. Finally, on the basis of the higher-order bound of the system \eqref{eq1.2}, we verify the a priori assumption to show the lifespan of the inviscid axially symmetric Hall-MHD system.
    
The following logarithm inequality is key to the higher-order estimate of the solution. We refer readers to (\cite{LZ2022}, Corollary 2.8) for a detailed proof.

\begin{lemma}{\label{inf nu cl}}
		for all divergence free vector field $\boldsymbol{\boldsymbol{g}}: \mathbb{R}^3 \rightarrow \mathbb{R}^3$ such that $\boldsymbol{\boldsymbol{g}} \in H^3\left(\mathbb{R}^3\right)$, the following estimate holds: 
		\begin{align*}
			\|\nabla \boldsymbol{g}(t, \cdot)\|_{L^{\infty}\left(\mathbb{R}^3\right)} \lesssim 1+\|\nabla \times \boldsymbol{g}(t, \cdot)\|_{B M O\left(\mathbb{R}^3\right)} \log \left(e+\|\boldsymbol{g}(t, \cdot)\|_{H^3\left(\mathbb{R}^3\right)}\right) .
		\end{align*}
		
	\end{lemma}

    \qed
	
As we have shown in the beginning of Section \ref{end of the1.1}, since the local well-posedness of the non-resistive axially symmetric Hall-MHD system in $H^m$ for $m\geq 3$, there exists $T_*>0$ such that: 
	\begin{equation}\label{invi prio}
        \begin{aligned}
        & t \sup _{0 \leq s \leq t} \left(\|\nabla \boldsymbol{h}(s,\cdot)\|_{L^{\infty}}+\|\nabla \mathcal{H}(s,\cdot)\|_{L^{\infty}}\right) \leq 1, \quad \text { for all } t \in\left[0, T_*\right).
    \end{aligned}
    \end{equation}
	In the following, our argument will be carried out before this $T_*$. Noticing that the estimates of $\na\boldsymbol{h}$ and $\na\mathcal{H}$ in Lemma \ref{nahH eges} are not influenced by the loss of viscous term, we still have the validity of \eqref{nabla h} and \eqref{nabla H}.

	\subsection{$L_t^{\infty} L^p $ -Estimate of $\Omega$ and $L_t^\infty L^{\infty}$ -Estimate of $w_\theta$}
	Our first step is to obtain an a priori bound for $\|\Omega\|_{L_t^{\infty} L^p} $. We will apply the $L^p$ energy estimate on the equation of $\Omega=\frac{w_\theta}{r}$ : 
	\begin{align}{\label{Omega eq1}}
		\partial_t \Omega+\boldsymbol{u} \cdot \nabla \Omega=-\partial_z \mathcal{H}^2 .
	\end{align}
The detailed result is stated as follows: 
	
	\begin{proposition}{\label{Ompro1}}
		Under the same assumptions as Theorem \ref{tivis}, the following estimate of $\Omega$ holds for all $1\leq p\leq \infty$:
		\begin{align*}
		\|\Omega(t,\cdot)\|_{L^p} \leq\left\|\Omega_0\right\|_{L^p}+2\|\mathcal{H}_0\|_{L^p},\q\text{for all}\q t\in[0, T_*)\,.
	\end{align*}
	\end{proposition}
    \begin{proof}
        Performing the $L^p$ estimates for \eqref{Omega eq1}, one arrives
	\begin{align}\label{EO1}
		\frac{d}{d t}\|\Omega(t,\cd)\|_{L^p}^p=-p \int_{\mathbb{R}^3} \partial_z \mathcal{H}^2 \cdot|\Omega|^{p-2} \Omega d x .
	\end{align}
Applying the Hölder inequality, we find 
	\begin{align*}
		p\left|\int_{\mathbb{R}^3} \partial_z \mathcal{H}^2 \cdot \left| \Omega\right|^{p-2} \Omega d x\right |&\leq p\left\|\partial_z \mathcal{H}^2(t,\cd)\right\|_{L^p}\|\Omega(t,\cd)\|_{L^p}^{p-1}\\
&\leq 2p\|\mathcal{H}(t,\cd)\|_{L^p}\|\nabla \mathcal{H}(t,\cd)\|_{L^{\infty}}\|\Omega(t,\cd)\|_{L^p}^{p-1}.
	\end{align*}
Recalling the $L^p$ conservation of $\mathcal{H}$ in Lemma \ref{L^pconsx}, one concludes from \eqref{EO1} that
	\begin{align*}
	\frac{d}{d t}\|\Omega(t,\cdot)\|_{L^p}^p \leq 2p\|\mathcal{H}_0\|_{L^p} \|\nabla \mathcal{H}(t,\cd)\|_{L^{\infty}} \|\Omega(t,\cd)\|_{L^p}^{p-1}.
	\end{align*}
Canceling $p\|\Omega(t,\cdot)\|_{L^p}^{p-1}$ on both sides and integrating over $[0,t]$, one deduces
\[
\|\O(t,\cd)\|_{L^p}\leq \|\O_0\|_{L^p}+2\|\mathcal{H}_0\|_{L^p}\int_0^t\|\nabla \mathcal{H}(s,\cd)\|_{L^{\infty}}ds.
\]
Thus one concludes the proposition by using the a priori condition \eqref{invi prio}.
    \end{proof}
    
	\begin{corollary}
	   Let the assumptions of Proposition \ref{Ompro1} be fulfilled. Then the following estimate of $w_\theta$ holds uniformly for all $t \in [0, T_*)$ : 
	\begin{align}{\label{infw_c}}
		\left\|w_\theta(t,\cdot)\right\|_{L^\infty} \leq  C \|(\boldsymbol{w_0},\boldsymbol{h_0})\|_{H^3} \exp \left(C\|(\boldsymbol{u_0},\mathcal{H}_0)\|_{H^3}t\right).
	\end{align}
		
	\end{corollary}
    \begin{proof}
    By using Lemma \ref{G-N} and Sobolev embedding, one has
        \begin{align*}
            \left\|\frac{u_r}{r}(t, \cdot)\right\|_{L^{\infty}} \leq C\left\|\nabla \frac{u_r}{r}(t, \cdot)\right\|_{L^2}^{1 / 2}\left\|\nabla \frac{u_r}{r}(t, \cdot)\right\|_{L^6}^{1 / 2} ,
        \end{align*}
        then by applying Lemma \ref{nav_r cl} and proposition \ref{Ompro1}, one arrives
	\begin{align*}
		\left\|\frac{u_r}{r}(t, \cdot)\right\|_{L^{\infty}} \leq C\|\Omega(t, \cdot)\|_{L^2}^{1 / 2}\|\Omega(t, \cdot)\|_{L^6}^{1 / 2} \leq C \left(\left\|\Omega_0\right\|_{L^2}+2\|\mathcal{H}_0\|_{L^2}\right)^{1 / 2}\left(\left\|\Omega_0\right\|_{L^6}+2\|\mathcal{H}_0\|_{L^6}\right)^{1 / 2},
        \end{align*}
    which indicates
        \begin{align*}
        \int_0^t\left\|\frac{u_r}{r}(s, \cdot)\right\|_{L^{\infty}} d s \leq C\|(\boldsymbol{u_0},\mathcal{H}_0)\|_{H^3}t.
	\end{align*}
    Using Lemma \ref{h_c w_c}, one deduces for all $t\in[0,T_*)$:
	\begin{align*}
		\begin{aligned}
        \left\|w_\theta(t,\cdot)\right\|_{L^\infty} & \leq \left(  \left\|\boldsymbol{w_0}\right\|_{L^\infty}+\|\boldsymbol{h_0}  \|_{L^\infty}\right)\exp \left(C\|(\boldsymbol{u_0},\mathcal{H}_0)\|_{H^3}t+C\right)\\
        &\leq  C \|(\boldsymbol{w_0},\boldsymbol{h_0})\|_{H^3} \exp \left(C\|(\boldsymbol{u_0},\mathcal{H}_0)\|_{H^3}t\right).
		\end{aligned}
	\end{align*}
This completes the proof.
    \end{proof}

	\subsection{$L^{\infty}$ Estimate of $\nabla \boldsymbol{u}$}
    In this subsection, we show an estimate of $\nabla \boldsymbol{u}$ by using Lemma \ref{inf nu cl} and the divergence-free condition $\nabla\cd \boldsymbol{u}=0$.
    \begin{lemma}
        For all $t < T_*$, the $L^{\infty} \text{ estimate}$ of $\nabla \boldsymbol{u}$ holds : 
    \begin{align}\label{infnabu}
		\|\nabla \boldsymbol{u}(t, \cdot)\|_{L^\infty} \lesssim 1+  \|(\boldsymbol{w_0},\boldsymbol{h_0})\|_{H^3} \exp \left(C\|(\boldsymbol{u_0},\mathcal{H}_0)\|_{H^3}t\right) \log \left(e+\|\boldsymbol{u}(t, \cdot)\|_{H^3}\right).
    \end{align}
    
    \end{lemma}
    \begin{proof}
        Applying Lemma \refeq{inf nu cl} to $\nabla \boldsymbol{u}$, for $\boldsymbol{u}$
	is divergence-free and belongs to $C\left(\left[0, T_*\right) ; H^3\left(\mathbb{R}^3\right)\right)$ we get 
	\begin{align*}
		\begin{aligned}
			\|\nabla \boldsymbol{u}(t, \cdot)\|_{L^\infty} & \lesssim 1+\|\nabla \times \boldsymbol{u}(t, \cdot)\|_{BMO} \log \left(e+\|\boldsymbol{u}(t, \cdot)\|_{H^3})\right. \\
			& \lesssim 1+\left\|w_\theta(t,\cd)\right\|_{L^\infty}\log \left(e+\|\boldsymbol{u}(t, \cdot)\|_{H^3}\right), \quad \forall t \in\left[0, T_*\right).
		\end{aligned}
	\end{align*}
Using (\refeq{infw_c}), we arrive at the $L^{\infty} \text{ estimate}$ of $\nabla \boldsymbol{u}$ \eqref{infnabu}.
    \end{proof}
    
	\subsection{End of the proof}
	
    Finally, we arrive at higher-order estimates of the non-resistive inviscid Hall-MHD system \eqref{eq1.2}. Define, for $t \in [0, T_*)$,
	\begin{align*}
		\begin{aligned}
			& E^2(t): =\|(\boldsymbol{u}, \boldsymbol{h}, \mathcal{H})(t, \cdot)\|_{H^3}^2,\q\text{with}\q E_0^2: =\left\|\left(\boldsymbol{u_0}, \boldsymbol{h_0}, \mathcal{H}_0\right)\right\|_{H^3}^2.
		\end{aligned}
	\end{align*} 
In this way, the higher-order bound we have given in Lemma \ref{nahH eges} reads 
    \begin{align*}
	E^2(t)\leq 	E^2_0 \exp\left( C \int_0^t\|\nabla(\boldsymbol{u},\boldsymbol{h},\mathcal{H})(s,\cd)\|_{L^{\infty}}ds\right).
	\end{align*} 
	Based on the estimate of $\|\nabla\boldsymbol{u}(t, \cdot)\|_{L^{\infty}}$ derived in \eqref{infnabu}, and recalling the a priori assumption \eqref{invi prio}, we have
\begin{small}
\[
E^2(t)\leq E_0^2 \exp \left( C \int_0^t\left\{ 1+  \|(\boldsymbol{w_0},\boldsymbol{h_0})\|_{H^3} \exp \left(C\|(\boldsymbol{u_0},\mathcal{H}_0)\|_{H^3}s\right) \log \left(e+\|\boldsymbol{u}(s, \cdot)\|_{H^3}\right) \right\}ds +C\right).
\]
\end{small}
This indicates
\begin{small}
\[
\begin{aligned}
\log (e + E^2(t)) \leq \log (e + E_0^2) + C + C \int_0^t \biggl\{ &1 + \|(\boldsymbol{w}_0, \boldsymbol{h}_0)\|_{H^3} \\
&\times \exp \bigl(C \|(\boldsymbol{u}_0, \mathcal{H}_0)\|_{H^3} s\bigr) \log \bigl(e + E^2(s)\bigr) \biggr\}  ds.
\end{aligned}
\]
\end{small}
Applying Grönwall's inequality, one finds that
\begin{align}{\label{eget}}
	E^2(t) \leq \left(C\left(e+E_0^2\right)\right)^{\exp \left(Ct + CE_0\exp \left(CE_0t\right)\right)}, \quad \forall t\in [0, T_*).
\end{align}
The last inequality follows from the a priori assumption \eqref{invi prio}. This shows that as long as \eqref{invi prio} holds on $\left[0, T_*\right)$, the solution $(\boldsymbol{u},\boldsymbol{h})$ to the initial value problem \eqref{eq1.2} would keep in $H^3$ before $t=T_*$.

Now it remains to verify the a priori assumption \eqref{invi prio}. Recalling the estimates of $\na\boldsymbol{h}$ and $\na\mathcal{H}$ in Lemma \ref{nahH eges}, using \eqref{eget}, one deduces
	\begin{align*}
			\|\nabla (\boldsymbol{h},\mathcal{H})(t, \cdot)\|_{L^{\infty}} \leq & \left\|\nabla (\boldsymbol{h}_0,\mathcal{H}_0)\right\|_{L^{\infty}} \exp \left(C \int_0^t\left(\|\nabla \boldsymbol{u}(s, \cdot)\|_{L^{\infty}}+\left\|\partial_z \mathcal{H}(s, \cdot)\right\|_{L^{\infty}}\right) d s\right)\\
			\leq & \left\|\nabla (\boldsymbol{h}_0,\mathcal{H}_0)\right\|_{L^{\infty}} \exp \left(C \int_0^t E(s) d s\right)\\
			\leq & \ve \exp \left(C t \left(C\left(e+E_0\right)\right)^{\exp \left(Ct + CE_0\exp \left(CE_0t\right)\right)}\right).
	\end{align*}
	Therefore, the following restriction of $T_*$
	\begin{align}{\label{meetprio}}
		\ve T_* \exp \left(C T_* \left(C\left(e+E_0\right)\right)^{\exp \left(CT_* + CE_0\exp \left(CE_0T_*\right)\right)}\right) \leq \frac{1}{3}
	\end{align}
	ensures the a priori assumption \eqref{invi prio} holds.  Without loss of generality, we assume that $\ve$ is sufficiently small, therefore $T_*>1$. Noticing that $\log (1+t) \leq t \leq e^t-1$ for all $t \geq 0$, one deduces
	\begin{align*}
		\exp (CT_* + CE_0\exp \left(CE_0T_*\right)) \leq \exp \left(\exp \left(CE_0 (T_*+1)\right)\right) ,
	\end{align*}
	and
	\begin{align*}
		C T_* \left(C\left(e+E_0\right)\right)^{B} &\leq \exp \big(B \log \left ( C T_* \left(C\left(e+E_0\right)\right)\right)\big)\\
        & \leq \exp \big( C B\left( (1+E_0)T_*\right)\big)
	\end{align*}
	for any $B>0$. Thus one derives that
    \begin{small}
	\begin{align*}
		\begin{aligned}
			& T_* \exp \left(C T_* \left(C\left(e+E_0\right)\right)^{\exp \left(CT_* + CE_0\exp \left(CE_0T_*\right)\right)}\right) 
			 \leq \frac{1}{3} \exp \left(\exp \left(\exp \left(\exp \left(C_*^{-1}\left(1+E_0\right) T_*\right)\right)\right)\right)
		\end{aligned}
	\end{align*}
    \end{small}
    for some constant $C_*>0$. Therefore, the condition \eqref{meetprio} is satisfied provided
	\begin{align*}
		\frac{1}{\ve}=\exp \left(\exp \left(\exp \left(\exp \left(C_*^{-1}\left(1+E_0\right)T_*\right)\right)\right)\right),
	\end{align*}
	for sufficiently small $\varepsilon>0$. In this way, one finds
	\begin{align*}
		T_*=\frac{C_*}{1+E_0} \log\left(\log \left(\log \left(\log \left(\ve^{-1}\right)\right)\right)\right).
	\end{align*}
	This gives the desired lower bound of the lifespan in Theorem \ref{tivis}. 
	
    \section*{Acknowledgment}
    T. Zhou is supported by Postgraduate Research \& Practice Innovation Program of Jiangsu Province(No. KYCX25\_1562).\\
\\
\textbf{Author Contributions} LY and TZ wrote the main manuscript text and reviewed the manuscript.\\
\\
\textbf{Data Availability} No datasets were generated or analysed during the current study.
\section*{Declarations}
\textbf{Conflict of interest} The authors declare no conflict of interest.

\vspace{0.5em}
Linbin Yang and Taoran Zhou\\
School of Mathematics and Statistics\\
Nanjing University of Information Science and Technology\\
Nanjing 210044\\
China\\
E-mails: yanglinbin@nuist.edu.cn;  zhoutaoran@nuist.edu.cn

    \end{document}